\documentclass{amsart}

\usepackage{amssymb}
\usepackage[hidelinks]{hyperref}
\usepackage[textsize=footnotesize]{todonotes}
\usepackage{tikz}
\usetikzlibrary{cd}
\usepackage{float}
\usepackage{graphics}

\theoremstyle{plain}
\newtheorem{theorem}{Theorem}[section]

\newtheorem{corollary}[theorem]{Corollary}
\newtheorem{proposition}[theorem]{Proposition}

\theoremstyle{definition}
\newtheorem{definition}[theorem]{Definition}

\newtheorem{example}[theorem]{Example}

\theoremstyle{remark}
\newtheorem{remark}[theorem]{Remark}

\newcommand{\bQ}{\mathbb{Q}}
\newcommand{\Q}{\bQ}

\newcommand{\cB}{\mathcal{B}}

\newcommand{\cI}{\mathcal{I}}
\newcommand{\I}{\cI}
\newcommand{\cJ}{\mathcal{J}}
\newcommand{\J}{\cJ}

\newcommand{\cP}{\mathcal{P}}

\newcommand{\cK}{\mathcal{K}}
\newcommand{\K}{\cK}

\newcommand{\cW}{\mathcal{W}}

\newcommand{\cc}{\mathfrak{c}}

\newcommand{\fin}{\mathrm{Fin}}
\newcommand{\Fin}{\mathrm{Fin}}
\newcommand{\ED}{\mathcal{ED}} 
\newcommand{\nwd}{\mathrm{NWD}}  
  
\newcommand{\conv}{\mathrm{CONV}}

\newcommand{\BW}{\mathrm{BW}} 
\newcommand{\FinBW}{\mathrm{FinBW}} 
\newcommand{\hFinBW}{\mathrm{hFinBW}} 
\newcommand{\hBW}{\mathrm{hBW}}

 \usepackage{gradient-text}

\begin{document}


\title{A new order for ideal sequential compactness}


\author[Adam Kwela]{Adam Kwela}
\address[Adam Kwela]{Institute of Mathematics\\ Faculty of Mathematics\\ Physics and Informatics\\ University of Gda\'{n}sk\\ ul.~Wita  Stwosza 57\\ 80-308 Gda\'{n}sk\\ Poland}
\email{Adam.Kwela@ug.edu.pl}
\urladdr{https://mat.ug.edu.pl/~akwela}

\author[Dorota Lesner]{Dorota Lesner}
\address[Dorota Lesner]{Institute of Mathematics\\ Faculty of Mathematics\\ Physics and Informatics\\ University of Gda\'{n}sk\\ ul.~Wita  Stwosza 57\\ 80-308 Gda\'{n}sk\\ Poland}
\email{Dorota.Lesner@phdstud.ug.edu.pl}


\date{\today}






\begin{abstract}
Let $\mathcal{I}$ be an ideal on $\omega$ and $X$ be a topological space. A sequence $(x_n)_{n\in \omega}$ in $X$ is $\mathcal{I}$-convergent if there is $x\in X$ such that $\{n\in \omega:x_n\notin U\}\in\mathcal{I}$ for every open neighborhood $U$ of $x$. We examine the following variant of sequential compactness associated with $\I$: $X$ is $\mathrm{BW}(\mathcal{I})$ if for every sequence  $(x_n)_{n\in \omega}$ in $X$ there is $A\notin\mathcal{I}$ such that $(x_n)_{n\in A}$ is $\mathcal{I}$-convergent. 

We introduce a new preorder on ideals, denoted $\leq_{BW}$, such that $\mathcal{I}\leq_{BW}\mathcal{J}$ implies that every $\mathrm{BW}(\mathcal{J})$ space is $\mathrm{BW}(\mathcal{I})$. Our main result states that under CH the above implication can be reversed in the case of $\mathbf{F_\sigma}$ ideals $\I$ and $\J$. 

We compare $\leq_{BW}$ with the Kat\v{e}tov order and study the relation $\leq_{BW}$ among some well-known ideals (e.g. the van der Waerden ideal $\mathcal{W}$ consisting of all subsets of $\omega$ that do not contain arbitrary long finite arithmetic progressions). As a consequence, we answer two open questions posed by Filip\'{o}w and Tryba in [Top. App. {\textbf{178}} (2014), 438--452] concerning comparison of $\mathrm{BW}(\mathcal{W})$ with the class of sequentially compact spaces.
\end{abstract}


\maketitle


\section{Introduction}

For the necessary definitions and notations, see Section \ref{sec:preliminaries}.

Let $\I$ be an ideal on $\omega$ and $A\in\I^+$. A sequence $(x_n)_{n\in A}$ in a topological space $X$ is said to be $\I$-convergent, if there is $x\in X$ such that $\{n\in A:x_n\notin U\}\in\I$ for every open neighborhood $U$ of $x$. 

\begin{definition}
Let $\I$ be an ideal on $\omega$. A topological space $X$ has the $\BW(\I)$ property if for every sequence  $(x_n)_{n\in \omega}$ in $X$ there is $A\in\I^+$ such that $(x_n)_{n\in A}$ is $\I$-convergent. In this case we briefly write $X\in\BW(\I)$ (however, formally $\BW(\I)$ is a class of topological spaces, not a set). 
\end{definition}

Obviously, $\BW(\Fin)$ is the class of sequentially compact spaces ($\Fin$ denotes the ideal consisting of all finite subsets of $\omega$). For this reason, the classes $\BW(\I)$ can be viewed as generalizations of sequential compactness. The aim of this paper is to study inclusions between the classes $\BW(\I)$. 

The case of $[0,1]$ is of particular interest, since the famous Bolzano-Weierstrass theorem states that the space $[0,1]$ is sequentially compact, i.e., that $[0,1]\in\BW(\Fin)$. Deep studies of ideals satisfying $[0,1]\in\BW(\I)$ have been performed by Filip\'{o}w, Mro\.{z}ek, Rec\l{}aw and Szuca in \cite{MR2320288}, where it is shown, among others, that $[0,1]\in\BW(\I)$ for every $\mathbf{F_\sigma}$ ideal $\I$. However, there are $\mathbf{F_{\sigma\delta}}$ ideals $\I$ with $[0,1]\notin\BW(\I)$ (see \cite{MR1181163}).

The classes $\BW(\I)$ have been studied previously. Below we present a short overview of the known results. Bernstein showed in \cite[Theorem 3.3]{MR251697} that every compact space has the $\BW(\I)$ property for every maximal ideal $\I$. In particular, for maximal ideals $\I$ it is possible to find a space with the $\BW(\I)$ property that is not sequentially compact. If $\I$ is an $\mathbf{F_\sigma}$ ideal then the class $\BW(\I)$ includes all compact metric spaces, all compact linearly ordered topological spaces and every limit ordinal of uncountable cofinality considered with the order topology (\cite[Corollary 5.7]{MR2961261}; see also \cite[theorem 1.1]{FT}).

There are also other ways of generalizing sequential compactness: $k$-Ramsey spaces (for $k\in\omega\setminus\{0\}$) and Hindman spaces, which are topological counterparts of $k$-dimensional Ramsey's theorem for coloring graphs and Hindman's finite sums theorem. It is worth noting that mathematicians studying such spaces come from different areas of mathematics such as ergodic theory (see \cite{MR2757532,BergelsonZelada,MR603625,MR531271}), topology (see \cite{MR3461173,MR3461181,MR2904054}) or set theory (see \cite{KojmanShelah}). Recently, Kubi\'{s} and Szeptycki in \cite{MR4552506} proved under CH that for each $k>1$ there is a $k$-Ramsey space which is not $(k+1)$-Ramsey. Later Coral, Guzm\'{a}n and L\'{o}pez-Callejas in \cite{High-dimensional} performed similar constructions under various different set-theoretic assumptions.

Another way of generalizing sequential compactness uses ideals on $\omega$.

\begin{definition}
Let $\I$ be an ideal on $\omega$. A topological space $X$ has the $\FinBW(\I)$ property if for every sequence $(x_n)_{n\in \omega}$ of elements of $X$ there is $A\in\I^+$ such that $(x_n)_{n\in A}$ is convergent (in the usual sense).
\end{definition}

In contrast to $\BW(\I)$, each $\FinBW(\I)$ is a subclass of sequentially compact spaces. Studies of such spaces have been initiated by Kojman in \cite{Kojman}, where he considered the class $\FinBW(\cW)$ given by the van der Waerden ideal $\cW$ consisting of all subsets of $\omega$ that do not contain arbitrarily long finite arithmetic progressions. Later Kojman and Shelah in \cite{KojmanShelah} performed under CH the first construction of a Mr\'{o}wka space with the $\FinBW(\cW)$ property, which is not Hindman. This construction inspired a long list of papers devoted to studies of the above mentioned generalizations of sequential compactness -- see for instance \cite{MR3097000, MR4356195, FT, MR2471564, MR2052425, MR4358658, MR4584767}. Among those papers we can also find the recent article \cite{FKK23}, presenting a unified approach to $\FinBW(\I)$ classes, Hindman spaces and $k$-Ramsey spaces. However, $\BW(\I)$ classes are not captured by this approach.

It is known that in general $\I\leq_K\J$ implies $\FinBW(\J)\subseteq\FinBW(\I)$, where $\leq_K$ denotes the Kat\v{e}tov order on ideals (see \cite[Corollary 10.2]{MR4584767}; for the definition of $\leq_K$ see Section \ref{sec:preliminaries}). Recently, the article \cite{MR4584767} established a deeper connection of the classes $\FinBW(\I)$ with the Kat\v{e}tov order. Namely, under CH the following holds for all $\mathbf{G_{\delta\sigma\delta}}$ ideals $\I$ and $\J$:
$$\I\leq_K\J\ \iff\ \FinBW(\J)\subseteq\FinBW(\I).$$
The space constructed in the proof (in order to show the implication $\I\not\leq_K\J\ \implies\ \FinBW(\J)\not\subseteq\FinBW(\I)$) is based on ideas developed by Kojman and Shelah in \cite{KojmanShelah}. This result has been generalized in \cite{FKK23} with the use of a new order introduced specifically for this purpose. However, studies of $\BW(\I)$ classes require a different approach, since every Mr\'{o}wka space has the $\BW(\I)$ property, for every ideal $\I$ (see \cite[Theorem 2.1]{FT} or \cite[Theorem 2.2]{MR3097000}).

Inspired by the above result, in this paper we introduce a new order on ideals, denoted $\leq_{BW}$, show that in general $\mathcal{I}\leq_{BW}\mathcal{J}$ implies $\mathrm{BW}(\mathcal{J})\subseteq\mathrm{BW}(\mathcal{I})$ (see Theorem \ref{thm:zawieranieBW}), while in the realm of $\mathbf{F_\sigma}$ ideals $\I$ and $\J$ the following holds under CH:
$$\I\leq_{BW}\J\ \iff\ \BW(\J)\subseteq\BW(\I)$$
(see Corollary \ref{cor:spaces_for_Fsigma}). It is worth mentioning that the space constructed in our proof (showing that $\I\not\leq_{BW}\J\ \implies\ \BW(\J)\not\subseteq\BW(\I)$) is completely different from the ones based on Kojman's and Shelah's proof from \cite{KojmanShelah}. We believe that for this reason our methods can significantly contribute to the further development of the discussed research topics.

Additionally, we examine basic properties of $\leq_{BW}$ (cf. Section \ref{sec:basic}), compare it with the Kat\v{e}tov order (Proposition \ref{prop:basic-properties}(d) and Remark \ref{rem:BWvsK}) and study the relation $\leq_{BW}$ among some well-known ideals (Section \ref{sec:porownania}). In particular, we answer two open questions posed by Filip\'{o}w and Tryba in \cite[Questions 2 and 3]{FT} by proving that $\BW(\cW)$ is a subclass of sequentially compact spaces, while under CH there is a sequentially compact space without the $\BW(\cW)$ property (see  Corollaries \ref{cor:idealy_powyzej_Fin}(a) and \ref{cor: spaces}(a)). We end with some open problems and directions for further research (Section \ref{sec:last}).


\section{Preliminaries}
\label{sec:preliminaries}

\subsection{Basics about ideals}

A family $\I$ of subsets of a set $M$ is called an ideal on $M$ if it is closed under taking subsets and finite unions of its elements. Additionally, we assume that ideals are proper ($\neq \mathcal{P}(M)$) and that all finite subsets of $M$ belong to $\I$. Note that $M=\bigcup\I$. In this paper we only consider ideals on countable sets. 

For an ideal $\I$ we denote by $\I^+=\{A\subseteq \bigcup\I: A\notin\I\}$ the family of all $\I$-positive sets. If $A\in\I^+$, then the restriction of $\I$ to $A$ is given by $\I|A=\I\cap\mathcal{P}(A)$. It is an ideal on $A$.

An ideal on a set $M$ is called Borel ($\mathbf{F_\sigma}$, etc.) if it is a Borel ($\mathbf{F_\sigma}$, etc., respectively) subset of $\cP(M)$ with the topology induced from $2^M$ by identifying subsets of $M$ with their characteristic functions ($2^M=\{0,1\}^M$ is considered with the product topology, where each space $2=\{0,1\}$ carries the discrete topology). The simplest example of an ideal is $\Fin$, which is equal to the family of all finite subsets of $\omega$. This ideal is $\mathbf{F_\sigma}$. By a result of Sierpi\'{n}ski, there are no $\mathbf{G_\delta}$ ideals. 

Let $\I$ and $\J$ be ideals. We say that:
\begin{itemize}
    \item $\I$ and $\J$ are isomorphic (written $\I\cong\J$) if there is a bijection $f:\bigcup\J\to \bigcup\I$ such that: 
    $$f^{-1}[B]\in \J \iff B\in \I$$
    for all $B\subseteq\bigcup\I$;
    \item $\I$ is below $\J$ in the Kat\v{e}tov order (written $\I\leq_K\J$) if there is an arbitrary function $f:\bigcup\J\to \bigcup\I$ such that $f[A]\in\I^+$ for all $A\in\J^+$. Equivalently, $\I$ is Kat\v{e}tov below $\J$ if there is a function $f:\bigcup\J\to\bigcup\I$ such that $f^{-1}[A]\in\J$ for all $A\in\I$.  
\end{itemize}
An ideal $\I$ is homogeneous if $\I|A\cong\I$ for all $A\in\I^+$.

All notions concerning ideals considered in this paper are invariant under isomorphisms. Therefore one can think about all ideals on countable sets as ideals on $\omega$ by identifying the domain set of the ideal with $\omega$ via a fixed bijection.  

An ideal $\I$ is tall if for every infinite $A\subseteq \bigcup\I$ there is an infinite $B\subseteq A$ with $B\in\I$. It is not difficult to see that $\I$ is tall if and only if $\I\not\leq_K\Fin$. We say that an ideal $\I$ is nowhere tall if $\I|A$ is not tall, for every $A\in\I^+$.

An ideal $\I$ is maximal if there is no other ideal $\J$ on $\bigcup\I$ such that $\J\supsetneq\I$. Equivalently, $\I$ is maximal if for each $A\subseteq\bigcup\I$ either $A\in\I$ or $\bigcup\I\setminus A\in\I$. We say that an ideal $\I$ is nowhere maximal if $\I|A$ is not maximal, for every $A\in\I^+$. It is known that maximal ideals do not have the Baire property and are non-measurable (see \cite{Sierpinski}). In particular, every Borel ideal is nowhere maximal (since restrictions of Borel ideals remain Borel). 

\subsection{Examples of ideals}

In the sequel, we will prove results about the following ideals:
\begin{itemize}
    \item The summable ideal is given by 
    $$\I_{1/n} = \left\{A\subseteq\omega: \sum_{n\in A}\frac{1}{n+1}<\infty\right\}.$$
    More generally, if $f:\omega\to[0,\infty)$ is such that $\sum_{n\in\omega}f(n)=\infty$, then 
    $$\I_{f} = \left\{A\subseteq\omega: \sum_{n\in A}f(n)<\infty\right\}$$
    is called a summable ideal. All such ideals are $\mathbf{F_{\sigma}}$. 
    \item The ideal $\nwd$ consisting of all $A\subseteq \Q\cap [0,1]$ that are nowhere dense (in the topology inherited from $\mathbb{R}$) is $\mathbf{F_{\sigma\delta}}$.
    \item The van der Waerden's ideal is defined by 
    $$\mathcal{W}= \{A \subseteq \omega : \exists_{n \in \omega} \textrm{ } A \textrm{  does not contain a finite arithmetic progression of length n} \}.$$
    It is an ideal by the van der Waerden's Theorem (see \cite{vanderWearden}). $\cW$ is $\mathbf{F_{\sigma}}$.
\end{itemize}
Ideals $\I_{1/n}$, $\nwd$ and $\cW$ are tall.

For $A\subseteq M\times N$ and $m\in M$, write $A_{(m)}=\{n\in N: (m,n)\in A\}$, i.e., $A_{(m)}$ is 
the vertical section of $A$ at the point $m$. For two ideals $\I$ and $\J$ on $M$ and $N$, respectively, define: 
\begin{itemize}
    \item $\I \oplus \J=\{(A \times \{0\}) \cup (B \times \{1\}): A \in \I\text{ and } B \in \J \}$. This is a proper ideal on $(M\times\{0\})\cup(N\times\{1\})$ even if only one of $\I$ and $\J$ is proper. Hence, we allow $\I\oplus\mathcal{P}(N)$ and $\mathcal{P}(M)\oplus\J$.
    \item $\I\otimes\J=\left\{A\subseteq M\times N: \{m\in M:A_{(m)}\notin\J\}\in\I\right\}$. This is a proper ideal even if one of $\I$ and $\J$ is equal to $\{\emptyset\}$. Hence, we allow $\I\otimes\{\emptyset\}$ and $\{\emptyset\}\otimes\J$.
\end{itemize}
A particular case of the latter is $\Fin^2 = \Fin\otimes\Fin$, which is an $\mathbf{F_{\sigma\delta\sigma}}$ ideal on $\omega^2$.

A map $\mu:\mathcal{P}(\omega)\to[0,\infty]$ is a submeasure on $\omega$ provided that $\mu(\emptyset)=0$, $\mu(\{n\})<\infty$ for all $n\in\omega$ and:
$$\forall_{A,B\subseteq\omega}\ \mu(A)\leq\mu(A\cup B)\leq \mu(A)+\mu(B)\text{.}$$
It is lower semicontinuous if additionally:
$$\forall_{A\subseteq\omega}\ \mu(A)=\lim_{n\in\omega}\mu(A\cap n)\text{.}$$
An ideal $\I$ on $\omega$ is called a generalized density ideal if there is a sequence $(\mu_n)_{n \in \omega}$ of lower semicontinuous submeasures on $\omega$ and a sequence of pairwise disjoint finite sets $(I_n)_{n \in \omega}$ such that: 
$$\I=\left\{A \subseteq \omega: \lim_{n \rightarrow \infty} \mu_n(A\cap I_n)=0\right\}\text{.}$$ 
All such ideals are $\mathbf{F_{\sigma\delta}}$. The best known member of this class is the ideal of sets of asymptotic density zero:
\[
\I_d=\left\{A\subseteq\omega:\lim_{n\to\infty}\frac{|A\cap n|}{n}=0\right\}\text{.}
\]
Generalized density ideals were introduced by Farah in \cite[Section 2.10]{MR1988247} (see also \cite{MR2254542}) and later used in different contexts (see for instance \cite{MR3436368, MR2320288, MR4404626, MR3968131}).

\subsection{P-like properties of ideals}

For any set $M$ and any two subsets $A,B\subseteq M$ write $A\subseteq^\star B$ if $A\setminus B$ is finite. We say that an ideal $\I$ is:
\begin{itemize}
\item $P^+$ if for every $\subseteq$-decreasing sequence $(A_n)_{n\in\omega}$ of elements of $\I^+$ there is $A\in\I^+$ such that $A\subseteq^\star A_n$ for all $n$,
\item $\widehat{P}^+$ if for every $\subseteq$-decreasing sequence $(A_n)_{n\in\omega}$ of elements of $\I^+$ there is $A\in\I^+$ such that $A\setminus A_n\in\I$ for all $n$.
\end{itemize}
Note that sometimes $\widehat{P}^+$ ideals are called $P^+(\I)$ ideals (see \cite{zbMATH07566290}). By \cite[Lemma 1.2]{MR748847}, every $\mathbf{F_\sigma}$ ideal is $P^+$. It is clear that $P^+$ implies $\widehat{P}^+$. 

Following \cite{MR2320288}, we say that an ideal $\I$ has the:
\begin{itemize}
    \item $\hFinBW$ property, if $[0,1]\in\FinBW(\I|A)$ for every $A\in\I^+$;
    \item $\hBW$ property, if $[0,1]\in\BW(\I|A)$ for every $A\in\I^+$.
\end{itemize}
Obviously, the $\hFinBW$ property implies the $\hBW$ property. If $\I$ is a $\widehat{P}^+$ ideal, then $\I$ has the $\hBW$ property (by \cite[Proposition 3.5]{MR2320288}; see also \cite[Proposition 2.13]{Ultrafilters} for an example of a non-$\widehat{P}^+$ ideal with the $\hBW$ property). It is also known that all $P^+$ ideals (so also all $\bf{F_\sigma}$ ideals) have the $\hFinBW$ property (by \cite[Proposition 2.8]{zbMATH06010584}, see also \cite[Proposition 3.4]{MR2320288}). We will need the following characterization (see \cite[Proposition 3.3]{MR2320288} and \cite[Proposition 2.9]{zbMATH06010584}).

\begin{theorem}\label{thm: hbw charakteryzacja}
An ideal $\I$ has the hBW property if and only if for every set $A\in \I^+$ and every family $\{A_s: s\in 2^{<\omega}\}$ of subsets of $A$ such that:
\begin{enumerate}
\item $A_{\emptyset}=A$;
\item $A_s=A_{s^{\frown} (0)} \cup A_{s^{\frown} (1)}$ for all $s\in 2^{<\omega}$;
\item $A_{s^{\frown}  (0)} \cap A_{s^{\frown} (1)}=\emptyset$ for all $s\in 2^{<\omega}$;
\end{enumerate}
there exist $x\in 2^{\omega}$ and $B\in \I^+$ such that $B\setminus A_{x\restriction n}\in \cI$ for each $n\in \omega$. 
\end{theorem}

\section{A new preorder for ideals}
\label{sec:basic}

In this section we introduce the relation $\leq_{BW}$, study its basic properties and compare it with the Kat\v{e}tov order.

\begin{definition}
Let $\I$ and $\J$ be ideals. We write $\I\leq_{BW}\J$ if there is $f:\bigcup\J\to\bigcup\I$ such that for every $A\in\J^+$ there is $B\subseteq f[A]$, $B\in\I^+$ such that given any $C$ with $C\cap B\in\I^+$ we have $A\cap f^{-1}[C]\in\J^+$.
\end{definition}

Next result summarizes some basic properties of $\leq_{BW}$. In Remark \ref{rem:BWvsK} we will see that the implication from Proposition \ref{prop:basic-properties}(d) cannot be reversed.

\begin{proposition}
\label{prop:basic-properties}
Let $\I$, $\J$ and $\mathcal{K}$ be ideals.
\begin{itemize}
    \item[(a)] The relation $\leq_{BW}$ is reflexive and transitive.
    \item[(b)] $\I\oplus\J\leq_{BW}\I$.
    \item[(c)] $\I\leq_{BW}\I \otimes \J$.
    \item[(d)] If $\I\leq_{BW}\J$, then $\I\leq_{K}\J$.
    \item[(e)] There is a $\leq_{BW}$-antichain of cardinality $\cc$ consisting of summable ideals.
    \item[(f)] The relation $\leq_{BW}$ is invariant over isomorphisms of ideals, i.e. if $\I'$ is isomorphic to $\I$ and $\J'$ is isomorphic to $\J$, then
    \[
    \I\leq_{BW}\J\ \iff\ \I'\leq_{BW}\J'.
    \]
    \item[(g)] If $\mathcal{K}\leq_{BW}\I$ and $\mathcal{K}\leq_{BW}\J$ then $\mathcal{K}\leq_{BW}\I \oplus \J$.
\end{itemize}
\end{proposition}

\begin{proof}
(a): Since reflexivity is obvious, we will only prove that $\leq_{\BW}$ is transitive. Suppose that $f$ witnesses $\I \leq_\BW\J$ and $g$ witnesses $\J \leq_\BW \K$. We will show that $\I \leq_\BW K$. Let us define function $h: \bigcup \K \rightarrow \bigcup
\I$ as $h(x)=f(g(x))$ for any $x \in \bigcup \K$. Fix any $A \in \K^+$. By assumption we know that there is $B^* \in \J^+$, $B^* \subseteq g[A]$ such that given any $C$ with $C\cap B^*\in\J^+$ we have $A\cap g^{-1}[C]\in\K^+$. Since $f$ witnesses $\I \leq_\BW\J$, there is $B \subseteq f[B^*]$, $ B \in \I^+$ such that given any $C$ with $C\cap B\in\I^+$ we have $B^*\cap f^{-1}[C]\in\J^+$. Then $B \subseteq f[g[A]]=h[A]$. Fix any $C$ with $C \cap B \in \I^+$. Then we have $B^* \cap f^{-1}[C]\in \J^+$ and that implies $A \cap h^{-1}[C] = A \cap g^{-1}[f^{-1}[C]] \in \K^+$.

\medskip

(b): Without loss of generality we may assume that both $\I$ and $\J$ are ideals on $\omega$.
Let $f:\omega \to \omega\times\{0,1\}$ be such that $f(i)=(i,0)$ for all $i \in \omega$. We will show that $f$ is a function that witnesses $\I \oplus \J\leq_\BW \I$. Fix any $A \in \I^+$. Then $f[A]=A\times\{0\}\in (\I\oplus\J)^+$. Let $B=f[A]$. Fix any $C$ with $C\cap B \in (\I\oplus\J)^+$. Then $\I^+\ni\{i\in\omega:(i,0)\in B\cap C\}=A\cap\{i\in \omega:(i,0)\in C\}=A\cap f^{-1}[C]$.
 
\medskip

(c): Without loss of generality we may assume that both $\I$ and $\J$ are ideals on $\omega$.
Let $f:\omega^2\to\omega$ be given by $f(i,j)=i$ for all $(i,j)\in\omega^2$. We will show that $f$ witnesses $\I\leq_{BW}\I\otimes\J$. Fix any $A\in(\I \otimes\J)^+$. Then $B=\{i\in\omega: A_{(i)}\in\J^+\}$ is $\I$-positive and contained in $f[A]$. Let $C\subseteq\omega$ be such that $C\cap B\in\I^+$. Then $A\cap f^{-1}[C]\supseteq \bigcup_{i\in C\cap B}(\{i\}\times A_{(i)})\in(\I \otimes\J)^+$ (as $A_{(i)}\in\J^+$ for all $i\in C\cap B$).

\medskip

(d): Straightforward.

\medskip

(e): By \cite[Theorem 1]{GuzmanMeza} (see also \cite[Corollary 3.6]{Nikodem}) there is a family $\{\I_\alpha:\alpha<\cc\}$ consisting of summable ideals such that $\I_\alpha\not\leq_K\I_\beta$ for all $\alpha,\beta<\cc$, $\alpha\neq\beta$. Thus, it suffices to apply item (d).

\medskip

(f): Straightforward.

(g): Suppose that $f:\bigcup\I\to\bigcup\mathcal{K}$ witnesses $\mathcal{K}\leq_{BW}\I$ and $g:\bigcup\J\to\bigcup\mathcal{K}$ witnesses $\mathcal{K}\leq_{BW}\J$. Define the function $h:(\bigcup\I\times\{0\})\cup(\bigcup\J\times\{1\})\to\bigcup\mathcal{K}$ by:
\[
h(n,i)=\begin{cases}
f(n), & \text{if }i=0,\cr
g(n), & \text{if }i=1,
\end{cases}
\]
for all $(n,i)\in (\bigcup\I\times\{0\})\cup(\bigcup\J\times\{1\})$. Let $A\in(\I\oplus\J)^+$. Then either $\{n\in\bigcup\I:(n,0)\in A\}\in\I^+$ or $\{n\in\bigcup\J:(n,1)\in A\}\in\J^+$. Without loss of generality we will assume that $A'=\{n\in\bigcup\I:(n,0)\in A\}\in\I^+$. Then there is $\mathcal{K}^+\ni B\subseteq f[A']\subseteq h[A]$ such that for every $C$ with $C\cap B\in\mathcal{K}^+$ we have $A'\cap f^{-1}[C]\in\I^+$, which gives us $A\cap h^{-1}[C]\supseteq (A'\cap f^{-1}[C])\times\{0\}\in(\I\oplus\J)^+$. Hence, $h$ is a witness for $\mathcal{K}\leq_{BW}\I \oplus \J$.
\end{proof}

Next result establishes a connection of $\leq_{BW}$ with the classes $\BW(\I)$, which are of our main interest. As we will see in Example \ref{ex:Fin^2}, the below implication cannot be reversed in general. However, we will be able to reverse it under additional assumptions (see Theorem \ref{thm:space}).

\begin{theorem}
\label{thm:zawieranieBW}
Let $\I$ and $\J$ be ideals. If $\I\leq_{BW}\J$, then $\BW(\J)\subseteq\BW(\I)$.
\end{theorem}

\begin{proof}
Without loss of generality we may assume that both $\I$ and $\J$ are ideals on $\omega$. Suppose that a space $X$ has the property $\BW(\J)$. We will show that $X$ has the $\BW(\I)$ property. Fix any sequence $(x_n)_{n \in \omega} \subseteq X$. By assumption of the theorem we know that there exists a function $f:\omega \rightarrow \omega$  witnessing $\I \leq_\BW \J$. Then $(x_{f(n)})_{n \in \omega}$ is also a sequence in $X$. Since $X \in \BW(\J)$, there is $A \in \J^+$ and $x \in X$ with $(x_{f(n)})_{n \in A} \overset{\J}\longrightarrow x$. Now, let $ f[A]\supseteq B \in \I^+$ be the set from the definition of $\I \leq_\BW \J$. To show that $(x_n)_{n \in B} \overset{\I}\longrightarrow x$ we need to fix an open neighborhood $U$ of $x$. Define $C=\{n \in B: x_n \notin U \}$.  Then 
$$A \cap f^{-1}[C]=A \cap \{n \in f^{-1}[B]: x_{f(n)} \notin U \} \subseteq \{n \in A: x_{f(n)} \notin U\} \in \J .$$ 
By the choice of $B$, we have $\{n \in B : x_n \notin U \}= C \cap B \in \I$. Hence, $(x_n)_{n \in B} \overset{\I}\longrightarrow x$. 
\end{proof}

\begin{proposition}
\label{prop:Fin_vs_tall}
Let $\I$ be an ideal.
\begin{itemize}
    \item[(a)] $\I\leq_{BW}\Fin$ if and only if $\I$ is not tall. In particular, if $\I$ is not tall, then $\BW(\I)$ contains all sequentially compact spaces. 
    \item[(b)] If $\I$ is nowhere tall, then $\I$ and $\Fin$ are $\leq_{BW}$-equivalent. In particular, if $\I$ is nowhere tall, then $\BW(\I)$ is the class of all sequentially compact spaces.
\end{itemize}
\end{proposition}

\begin{proof}
(a): The implication $\impliedby$ follows from Proposition \ref{prop:basic-properties}(b), since any non-tall ideal is isomorphic to $\Fin\oplus\J$ for some ideal $\J$. On the other hand, the implication $\implies$ follows from Proposition \ref{prop:basic-properties}(d), since $\I\not\leq_K\Fin$ for tall ideals $\I$. The second part follows from Theorem \ref{thm:zawieranieBW} and the fact that $\BW(\Fin)$ is the class of all sequentially compact spaces.

(b): Let $\I$ be nowhere tall. Then $\I\leq_{BW}\Fin$ by item (a). We will show that $\Fin\leq_{BW}\I$. The second part will then follow from Theorem \ref{thm:zawieranieBW} and the fact that $\BW(\Fin)$ is the class of all sequentially compact spaces.

Without loss of generality we may assume that $\I$ is an ideal on $\omega$. Let $f:\omega \to \omega$ be the identity function. Fix $A \in \I^+$. Since $\I$ is nowhere tall, we can find an infinite $B\subseteq A$ such that $\I|B=\Fin|B$. Then $B = f[B]\subseteq f[A]$ and $B \in \fin^+$. Fix $C$ such that $C\cap B \in \fin^+$ and observe that then $A \cap f^{-1}[C]=A \cap C \supseteq B\cap C\in \I^+$.
\end{proof}

In Remark \ref{rem:nowhere tall vs non-tall} we discuss possible strengthenings of Proposition \ref{prop:Fin_vs_tall}(b).

\section{The new preorder among some well-known ideals}
\label{sec:porownania}

In this section we examine $\leq_{BW}$ among some of the ideals introduced in Section \ref{sec:preliminaries}. Firstly, we give examples of ideals that are $\leq_{BW}$-above $\Fin$.

\begin{proposition}\
\label{prop:idealy_powyzej_Fin}
\begin{itemize}
    \item[(a)] $\Fin\leq_{BW}\mathcal{W}$. 
    \item[(b)] $\Fin\leq_{BW}\Fin^2$. 
    \item[(c)] $\fin \leq_{BW} \I$, for every generalized density ideal $\I$. In particular, $\Fin\leq_{BW}\I_d$.
\end{itemize}
\end{proposition}

\begin{proof}
(a): By \cite[Example 2.6]{KT}, $\cW$ is homogeneous. Therefore, it is enough to show that there exists $E \in \cW^+$ such that $\Fin \leq_{\BW} \cW|E$ (see Proposition \ref{prop:basic-properties}(f)).   Let $E_n=\{10^n, 10^n+1, \ldots, 10^n +n  \}   $ for all $n \in \omega$ and $E=\bigcup_{n \in \omega}E_n\in\cW^+$. Now, let the function $f: E \rightarrow \omega$ be such that:
\[ f(i) = n\ \iff\ i \in E_n\]
for all $i \in E$. Fix $A \in(\cW|E)^+$ and for all $n \in \omega$ take arithmetic progression $A_n$ of length $n$ such that $A_n \subseteq A$. 

For all $n > 2$ take $k_n$ such that $A_n \subseteq E_{k_n}$. This is possible. Indeed, suppose that there is $n \in \omega$ for which $A_n \nsubseteq E_i$ for all $i \in \omega$. Let $(a_j)_{j<n}$ be the increasing enumeration of $A_n$ and $r>0$ be the common difference of $A_n$. We can find the smallest $k$ for which there are $l_1,l_2\in \omega$ such that $l_1<l_2$, $a_k \in E_{l_1}$ and $a_{k+1} \in E_{l_2}$. Then $r\geq10^{l_2}-10^{l_1}-l_1\geq 10^{l_1+1}-10^{l
_1}-l_1\geq 10^{l_1}+l_1$. Hence, $k=0$. There are two possible cases:
\begin{itemize}
\item If $a_{k+2}=a_2 \in E_{l_2}$, we have $a_2-a_1\leq 10^{l_2}+l_2-10^{l_2}=l_2$ and $a_1-a_0\geq10^{l_2}-10^{l_1}-l_1>10^{l_2-1}\geq l_2$. A contradiction with the fact that $A_n$ is an arithmetic progression. 
\item Now we have to consider the case in which $a_2 \notin E_{l_2}$. Then $a_2-a_1\geq10^{l_2+1}-10^{l_2}-l_2>10^{l_2}+l_2$ and $a_1-a_0\leq a_1\leq 10^{l_2}+l_2$. Just like in the previous case, it leads us to a contradiction with the fact that $A_n$ is an arithmetic progression.
\end{itemize}
Therefore, for every $n > 2$ it is possible to find $k_n$ such that $A_n \subseteq E_{k_n}$.

Define $B=\{k_i:i>2\}$. Clearly, such $B$ is infinite and contained in $f[A]$. Take any $C \subseteq \omega$ with $C \cap B \in \fin^+$. Then $f^{-1}[C] \supseteq \bigcup\{E_{k_n} : k_n \in C\cap B\}$ and hence $f^{-1}[C]\cap A \supseteq \bigcup\{A_n: k_n \in C \cap B \} \in (\cW|E)^+$.

\medskip

(b): Follows from Proposition \ref{prop:basic-properties}(c).

\medskip

(c): If $\I$ is a generalized density ideal, there is a sequence of lower semicontinuous submeasures $(\mu_n)_{n \in \omega}$ and a sequence of pairwise disjoint finite sets $(I_n)_{n \in \omega}$ such that $\I=\{A \subseteq \omega: \lim_{n \rightarrow \infty} \mu_n(A\cap I_n)=0\}$. Without loss of generality we may assume that $\bigcup_{n\in\omega}I_n=\omega$ (by adding, if necessary, to the sequences $(I_n)_{n \in \omega}$ and $(\mu_n)_{n \in \omega}$ new elements $I'_b$ and $\mu'_b$, for $b\in\omega\setminus\bigcup_{n\in\omega}I_n$, such that $I'_b=\{b\}$ and $\mu'_b(A)=0$ for all $A\subseteq\omega$). Let $f:\omega \rightarrow \omega$ be given by 
\[ f(k)=n\ \iff\ k \in I_n.\]
We will prove that such function $f$ witnesses $\fin \leq_{BW} \I$. Fix any $A \in \I^+$. Then $\lim_{n \rightarrow \infty}\mu_n(A\cap I_n) \neq 0$. Hence, there exist $k \in \omega$ and an increasing sequence of natural numbers $(n_i)_{i \in \omega}$ satisfying $\mu_{n_i}(A\cap I_{n_i})>\frac{1}{k}$ for all $i \in \omega$. Now, let $B=\{n_i:i \in \omega\}$. Clearly, $\fin^+\ni B\subseteq f[A]$. Fix any $C \subseteq \omega$ that satisfies $A\cap f^{-1}[C] \in \I$. Consequently, there can be only finitely many $i \in \omega$ for which $n_i \in C$. Therefore, $C\cap B$ is finite.
\end{proof}

\begin{corollary}\
\label{cor:idealy_powyzej_Fin}
    \begin{itemize}
        \item[(a)] Every $\BW(\cW)$ space is sequentially compact.
        \item[(b)] Every $\BW(\fin^2)$ space is sequentially compact.
        \item[(c)] If $\I$ is a generalized density ideal, then every $\BW(\I)$ space is sequentially compact. In particular, every $\BW(\I_d)$ space is sequentially compact.
    \end{itemize}
\end{corollary}

\begin{proof}
Follows from Theorem \ref{thm:zawieranieBW}, Proposition \ref{prop:idealy_powyzej_Fin} and the fact that $\BW(\Fin)$ is the class of all sequentially compact spaces.
\end{proof}

\begin{remark}
Filip\'{o}w and Tryba in \cite[Question 3]{FT} asked whether there is a $\BW(\mathcal{W})$ space which is not sequentially compact. Corollary \ref{cor:idealy_powyzej_Fin}(a) gives a negative answer.
\end{remark}

Next example shows that the implication from Theorem \ref{thm:zawieranieBW} (i.e., $\I\leq_{BW}\J\ \implies\ \BW(\J)\subseteq\BW(\I)$) cannot be reversed in general.

\begin{example}
\label{ex:Fin^2}
$\Fin^2\not\leq_{BW}\Fin$, but $\BW(\Fin)=\BW(\Fin^2)$. 
\end{example}

\begin{proof}
$\Fin^2\not\leq_{BW}\Fin$ follows from Proposition \ref{prop:basic-properties}(d) and the fact that $\Fin^2\not\leq_{K}\Fin$ (as $\Fin^2$ is a tall ideal). The inclusion $\BW(\Fin)\supseteq\BW(\Fin^2)$ is proved in Corollary \ref{cor:idealy_powyzej_Fin}(b). Thus, we only need to show that $\BW(\Fin)\subseteq\BW(\Fin^2)$.

Let $X$ be a sequentially compact space and fix any sequence $f:\omega^2\to X$. Then for each $n\in\omega$ we can find $M_n\in\fin^+$ and $x_n\in X$ such that $(f(n,i))_{i\in M_n}$ converges to $x_n$. Moreover, there are $M\in\Fin^+$ and $x\in X$ such that $(x_n)_{n\in M}$ converges to $x$. Define $A=\bigcup_{n\in M}\{n\}\times M_n$. Then $A\in(\fin^2)^+$. We will show that $(f(i,j))_{(i,j)\in A}$ is $\fin^2$-convergent to $x$. 

Let $U$ be an open neighborhood of $x$. Then there is $K\in\fin$ such that $x_n\in U$ for all $n\in M\setminus K$, i.e., $U$ is an open neighborhood of each $x_n$ with $n\in M\setminus K$. Thus, for each $n\in M\setminus K$ there is $F_n\in\fin$ such that $f(n,i)\in U$ for all $i\in M_n\setminus F_n$. Observe that
\[
B=(K\times\omega)\cup\bigcup_{n\in M\setminus K}(\{n\}\times F_n)\in\Fin^2
\]
and $f(n,i)\in U$ for all $(n,i)\in A\setminus B$.
\end{proof}

Now we give examples of ideals that are not $\leq_{BW}$-above $\Fin$.

\begin{proposition}\
\label{prop:idealy_nie_powyzej_Fin}
\begin{itemize}
    \item[(a)] $\Fin\not\leq_{BW}\I_{1/n}$. 
    \item[(b)] $\Fin\not\leq_{BW}\ED_-$, where $\ED_-=\{A\subseteq\omega^2: \exists_{k\in\omega}\ \forall_{n\in\omega}\ |A_{(n)}|\leq k\}$. 
    \item[(c)] $\Fin\not\leq_{BW}\I$ for every maximal ideal $\I$. 
    \item[(d)] $\Fin\not\leq_{BW}\nwd$.
\end{itemize}
\end{proposition}

\begin{proof}
(a): Take any function $f : \omega\rightarrow \omega$ and denote $S_n= f^{-1}[\{n\}]$ for all $n \in \omega$. In case there is $n \in \omega$ with $S_n\in \I_{1/n}^+ $, the set $S_n$ is the one we need (as there is no $\Fin^+\ni B\subseteq \{n\}=f[S_n]$). Let us now consider the case when all $S_n \in \I_{1/n}$. Define numbers $p_n= \sum_{i \in S_n}\frac{1}{i+1}<\infty$ for all $n \in \omega$. We will consider two cases. 

\smallskip

\textit{Case 1:} $\lim_{n \rightarrow \infty} p_n=0$. Let $A=\omega$. Fix any $\Fin^+ \ni B\subseteq f[A]$. Now, inductively for each $n\in\omega$ take $c_n\in B\setminus\{c_i:i<n\}$ such that $p_{c_n} \leq \frac{1}{2^n}$. Define $C=\{ c_n:n\in\omega \}$. For that $C$ it holds that $C\cap B=C \in \fin^+$ and $A\cap f^{-1}[C]=f^{-1}[C]=\bigcup_{n\in\omega}S_{c_n}\in \I_{1/n}$.

\smallskip

\textit{Case 2:} $\lim_{n \rightarrow\infty}p_n \neq 0$. Then we can find a set $Z\in\Fin^+$ and $p>0$ such that $p_{z}>p$ for all $z\in Z$. Let $(z_n)_{n\in\omega}$ be the increasing enumeration of $Z$. Take $k\in\omega$ such that $\frac{1}{k+1}<p$. Observe that it is possible to find a set $A\subseteq\bigcup_{n \geq k}S_{z_n}$ such that  $$r_n=\sum_{i \in S_{z_{n}}\cap A}\frac{1}{i+1}\geq \frac{1}{n+1}$$ for all $n \geq k$ and moreover $\lim_{n \to \infty}r_n=0$. Clearly, such $A$ belongs to $\I^+_{1/n}$. Now, fix $\fin^+\ni B\subseteq f[A]\subseteq Z$. Similarly as in the previous case, inductively for each $n \in \omega$ take $c_n\in B\setminus\{c_i:i<n\}$ such that $$\sum_{i \in S_{c_n}\cap A}\frac{1}{i+1}\leq \frac{1}{2^n}\text{.}$$ Define $C=\{c_n:n \in\omega\}$. Then $C\cap B=C \in \fin^+$ and $f^{-1}[C]\cap A =\bigcup_{n \in \omega}S_{c_n} \cap A \in \I_{1/n}$.

\medskip

(b): Fix any $f:\omega^2\to\omega$. If $f^{-1}[\{n\}]\notin\ED_-$ for some $n\in\omega$, then put $A=f^{-1}[\{n\}]$ and observe that there is no $B\subseteq f[A]=\{n\}$ with $B\in\Fin^+$, so there is nothing to check. Thus, we can assume from now on that $f^{-1}[\{n\}]\in\ED_-$ for all $n\in\omega$. 

Inductively pick $k^n_i\in\omega$ for all $n\in\omega$ and $i\leq n$ such that $f(n,k^n_i)\neq f(m,k^m_j)$ whenever $(n,i)\neq(m,j)$. Observe that this is possible. Indeed, given $n\in\omega$ and $i\leq n$, the set 
\[
E^n_i=\{f(m,k^m_j): m<n,j\leq m\}\cup\{f(n,k^n_j): j<i\}
\]
is finite, so $f^{-1}[E^n_i]\in\ED_-$, which gives that $(\{n\}\times\omega)\setminus f^{-1}[E^n_i]\neq\emptyset$. 

Define $A=\{(n,k^n_i): n\in\omega,i\leq n\}$. Then $A\in(\ED_-)^+$ as $|A_{(n)}|=n+1$ for all $n\in\omega$. Fix any infinite $B\subseteq f[A]$. Since the sets $\{f(n,k^n_i):i\leq n\}$ are finite, we can find an infinite $C\subseteq B$ such that $|C\cap \{f(n,k^n_i):i\leq n\}|\leq 1$ for all $n\in\omega$. Then $B\cap C=C\in\Fin^+$, but $A\cap f^{-1}[C]\in\ED_-$, as $|(A\cap f^{-1}[C])_{(n)}|\leq 1$ for all $n\in\omega$. 

\medskip

(c): Let $\I$ be a maximal ideal and assume towards contradiction that $\Fin\leq_{BW}\I$. Then $\BW(\I)\subseteq\BW(\Fin)$ (by Theorem \ref{thm:zawieranieBW}). However, $\BW(\I)$ contains all compact spaces (by \cite[Theorems 3.1 and 3.3]{MR251697}; see also \cite[Theorem 1.2]{FT}) and there are compact spaces which are not sequentially compact (i.e., are not $\BW(\Fin)$), a contradiction. 

(d): Let $f:\mathbb{Q}\cap[0,1]\to\omega$ be arbitrary. We need to find $A\in\nwd^+$ such that for all $B\subseteq f[A]$, $B\in\Fin^+$ there is an infinite $C\subseteq B$ such that $A\cap f^{-1}[C]\in\nwd$. 

If $f^{-1}[\{n\}]\in\nwd^+$ for some $n\in\omega$, it suffices to put $A=f^{-1}[\{n\}]$. Thus, we can assume that $f^{-1}[\{n\}]\in\nwd$ for all $n\in\omega$.

Let $\{U_n:n\in\omega\}$ be an enumeration of all non-empty open sub-intervals of $[0,1]$ with rational endpoints. Inductively find a sequence $(x_n)_{n\in\omega}$ such that $x_n\in\mathbb{Q}\cap U_n$ and $f(x_n)\neq f(x_m)$ for all $n,m\in\omega$, $n\neq m$. This is possible as given any $n\in\omega$, the set $f^{-1}[\{f(x_m):m<n\}]$
is nowhere dense. Put $A=\{x_n:n\in\omega\}$. Then $A\in\nwd^+$ as it is dense in $[0,1]$.

Fix any $B\subseteq f[A]$, $B\in\Fin^+$. Then there is an infinite $D\subseteq \{n\in\omega: f(x_n)\in B\}$ such that $(x_n)_{n\in D}$ is convergent. Put $C=\{f(x_n):n\in D\}$. Then $C\subseteq B$, $C$ is infinite and $A\cap f^{-1}[C]=\{x_n:n\in D\}\in\nwd$.
\end{proof}

\begin{remark}
\label{rem:BWvsK}
Note that the implication from Proposition \ref{prop:basic-properties}(d) cannot be reversed, i.e., there are ideals $\I$ and $\J$ such that $\I\leq_K\J$, but $\I\not\leq_{BW}\J$. Indeed, this is the case for $\I=\Fin$ and $\J=\I_{1/n}$ (by Proposition \ref{prop:idealy_nie_powyzej_Fin}(a)). Other items of Proposition \ref{prop:idealy_nie_powyzej_Fin} provide more similar examples.
\end{remark}

\begin{remark}
Among tall $\mathbf{F_\sigma}$ ideals there are ideals $\leq_{BW}$-above $\Fin$ (Proposition \ref{prop:idealy_powyzej_Fin}(a)) as well as not $\leq_{BW}$-above $\Fin$ (Proposition \ref{prop:idealy_nie_powyzej_Fin}(a)). Similarly, among non-tall $\mathbf{F_\sigma}$ ideals there are ideals $\leq_{BW}$-above $\Fin$ (for instance $\Fin$) as well as not $\leq_{BW}$-above $\Fin$ (Proposition \ref{prop:idealy_nie_powyzej_Fin}(b)).
\end{remark}

\begin{remark}
\label{rem:nowhere tall vs non-tall}
By Proposition \ref{prop:Fin_vs_tall}(b), every nowhere tall ideal is $\leq_{BW}$-equivalent to $\Fin$. Observe that there are other ideals $\leq_{BW}$-equivalent to $\Fin$. Indeed, the ideal $\mathcal{W}\oplus\Fin$ is a good example, since it is not nowhere tall (its restriction to $\omega\times\{0\}$ is tall), but it is $\leq_{BW}$-equivalent to $\Fin$ ($\mathcal{W}\oplus\Fin\leq_{BW}\Fin$ follows from Proposition \ref{prop:basic-properties}(b), while $\Fin\leq_{BW}\mathcal{W}\oplus\Fin$ follows from Propositions \ref{prop:basic-properties}(a), \ref{prop:basic-properties}(g) and \ref{prop:idealy_powyzej_Fin}(a)).

On the other hand, not every non-tall ideal is $\leq_{BW}$-equivalent to $\Fin$. For example $\ED_-$ is a non-tall ideal, but as it is shown in Proposition \ref{prop:idealy_nie_powyzej_Fin}(b) it is not $\leq_\BW$-above $\fin$.
\end{remark}

\begin{proposition}
\label{prop:WvsSummable}
$\mathcal{W}\not\leq_{BW}\I_{1/n}$ and $\I_{1/n}\not\leq_{BW}\mathcal{W}$.
\end{proposition}

\begin{proof}
$\I_{1/n}\not\leq_{BW}\mathcal{W}$ follows from Proposition \ref{prop:basic-properties}(d) as $\I_{1/n}\not\leq_{K}\mathcal{W}$ (see \cite[Proposition 4.1]{MR4797304}). On the other hand, $\mathcal{W}\leq_{BW}\I_{1/n}$ would imply $\Fin\leq_{BW}\I_{1/n}$ (by Propositions \ref{prop:basic-properties}(a) and \ref{prop:idealy_powyzej_Fin}(a)), which contradicts Proposition \ref{prop:idealy_nie_powyzej_Fin}(a).
\end{proof}


\section{Main result}
\label{sec:main}

In this section, we prove our main result.

\begin{theorem}
\label{thm:space}
Assume CH. Let $\J$ be a nowhere maximal $P^+$ ideal and $\I$ be an ideal with the property $\hBW$. Then the following are equivalent:
\begin{itemize}
    \item[(a)] $\J\not\leq_{BW}\I$,
    \item[(b)] there is a $\BW(\I)$ space which is not $\BW(\J)$,
    \item[(c)] there is a Hausdorff non-compact $\BW(\I)$ space which is not $\BW(\J)$.
\end{itemize}
\end{theorem}

\begin{proof}
(c)$\implies$(b): Obvious.

\medskip

(b)$\implies$(a): Follows from Theorem \ref{thm:zawieranieBW}.

\medskip

(a)$\implies$(c): For every $n\in\omega$ define $x_n\in 2^\omega$ and $h_n:2^\omega\to 2$ by:
\[
x_n(i)=\begin{cases}
0, & \text{if }i\neq n,\cr
1, & \text{if }i=n,
\end{cases}
\]
and $h_n(x)=x(n)$ for all $x\in 2^\omega$.

Let $\{f_\alpha: \omega\leq\alpha<\omega_1\}$ be such an enumeration of the family $\{f\in\omega^\omega_1: \forall_{\beta<\omega_1}\ f^{-1}[\{\beta\}]\in\I\}$ that $f_\alpha[\omega]\subseteq\alpha$ for all $\alpha$. Enumerate also $\J^+=\{A_\alpha: \omega\leq\alpha<\omega_1\}$ and denote $\nabla=\{\alpha\in[\omega,\omega_1):f_\alpha^{-1}[[\omega,\omega_1)]]\in\I\}$.

Using transfinite induction, we will construct $\{x_\alpha\in 2^\omega: \omega\leq\alpha<\omega_1\}$, $\{E_\alpha\subseteq\omega: \omega\leq\alpha<\omega_1\}$, $\{Z_\alpha\subseteq\omega: \omega\leq\alpha<\omega_1\}$ and $\{h_\alpha\in 2^{2^\omega}: \omega\leq\alpha<\omega_1\}$ satisfying the following properties for every $\omega\leq\alpha<\omega_1$: 
\begin{itemize}
    \item[(P1)] $E_\alpha\subseteq A_\alpha$, $E_\alpha\in\J^+$, $A_\alpha\setminus E_\alpha\in \J^+$,
    \item[(P2)] $x_\alpha(i)=1\ \iff\ i\in E_\alpha$,
    \item[(P3)] $(h_{f_\beta(n)}(x_\alpha))_{n\in Z_\beta}$ is $\I$-convergent to $0=h_\beta(x_\alpha)$, for all $\omega\leq\beta\leq\alpha$,
    \item[(P4)] $x_\alpha\neq x_\beta$ for all $\beta<\alpha$,
    \item[(P5)] $Z_\alpha\in\I^+$ and:
    \begin{itemize}
        \item if $\alpha\notin\nabla$ then $f_\alpha[Z_\alpha]\subseteq[\omega,\alpha)$,
        \item if $\alpha\in\nabla$ then $f_\alpha[Z_\alpha]\subseteq\omega$ and for every $B\in\J^+$, $B\subseteq f_\alpha[Z_\alpha]$ there is $C$ such that $C\cap B\in\J^+$ and $Z_\alpha\cap f_\alpha^{-1}[C]\in\I$,
    \end{itemize}
    \item[(P6)] $(h_{f_\alpha(n)}(x_\beta))_{n\in Z_\alpha}$ is $\I$-convergent to $h_\alpha(x_\beta)$, for all $\beta<\alpha$,
    \item[(P7)] $h_\alpha\restriction 2^\omega\setminus\{x_\beta: \beta<\alpha\}=0$ and $h_\alpha\restriction \{x_n: n\in\omega\}=0$.
\end{itemize}

\medskip

Before presenting the construction, we show how to finish the proof using the above properties. 

Define $X=\{h_\alpha: \alpha<\omega_1\}$ and $\tilde{X}=\{x_\alpha:\alpha<\omega_1\}$. For each $h\in X$ let 
\[
\cB_h=\left\{\left\{g\in X: \forall_{y\in G}\ g(y)=h(y)\right\}: G\in[\tilde{X}]^{<\omega}\right\}.
\]
It is easy to check that $\{\cB_h: h\in X\}$ is a neighborhood system, so it generates a topology on $X$ (\cite[Proposition 1.2.3]{MR1039321}). Consider $X$ as a topological space with this topology. We need to show that:
\begin{itemize}
    \item[(i)] $X\notin\BW(\J)$,
    \item[(ii)] $X\in\BW(\I)$,
    \item[(iii)] $X$ is Hausdorff,
    \item[(iv)] $X$ is not compact.
\end{itemize}

(i): We will show that $X$ does not have the $\BW(\J)$ property. Consider the sequence $(h_n)_{n \in \omega}$. Suppose to the contrary that there are $A\in \J^+$ and $\beta<\omega_1$ such that $(h_n)_{n \in A} \overset{\J}{\longrightarrow}h_\beta$. Then there is $\alpha$ for which $A=A_\alpha$. The set $V=\{g \in X: g(x_\alpha)=h_\beta(x_\alpha)\}$ is an open neighborhood of $h_\beta$. We have two possible cases: 

\smallskip

\textit{Case 1:} $h_\beta(x_\alpha)=0$. Then $\{ n \in A: h_n\notin V\}=\{n \in A_\alpha: h_n(x_\alpha)\neq h_\beta(x_\alpha)\}=\{n \in A_\alpha: h_n(x_\alpha)\neq 0\}=A_\alpha\cap E_\alpha=E_\alpha$, by (P1) and (P2).

\smallskip

\textit{Case 2:} $h_\beta(x_\alpha)=1$. Then $\{ n \in A: h_n\notin V\}=\{n \in A_\alpha: h_n(x_\alpha)\neq h_\beta(x_\alpha)\}=\{n \in A_\alpha: h_n(x_\alpha)\neq 1\}=A_\alpha \setminus E_\alpha$, by (P2).

\smallskip

It follows from (P1) that both $E_\alpha$ and $A_\alpha \setminus E_\alpha$ are $\J$-positive, which contradicts the assumption that $(h_n)_{n \in A} \overset{\J}{\longrightarrow}h_\beta$.

\medskip

(ii): Now we prove that $X$ is $\BW(\I)$. Fix any sequence $(g_n)_{n \in \omega}$ in $X$. Let $f\in \omega_1^\omega$ be a function such that $g_n=h_{f(n)}$ for all $n \in \omega$. If there is $\beta<\omega_1$ for which $f^{-1}[\{\beta\}] \in \I^+$, then $(g_n)_{n\in f^{-1}[\{\beta\}]}$ is $\I$-convergent to $h_{\beta}$ and we are done. Otherwise, there is $\omega \leq \alpha < \omega_1$ such that $f=f_\alpha$. We will prove that $(g_n)_{n \in Z_\alpha} \overset{\I}{\longrightarrow}h_\alpha$ (as $Z_\alpha\in\I^+$ by (P5), this will finish the proof of $X\in\BW(\I)$). For this purpose, let us fix a basic open neighborhood $V$ of $h_\alpha$. Then $$V=\{h \in X: \forall_{i \leq m} h(x_{\beta_i})=h_\alpha(x_{\beta_i})\}$$ for some $m \in \omega$ and $\beta_0, \ldots,\beta_m < \omega_1$. Hence, we have 
\begin{equation*}
    \begin{split}
        \{n \in Z_\alpha: g_n\notin V\}
        &
         =\left\{n \in Z_\alpha: \exists_{i \leq m}g_n(x_{\beta_i})\neq h_\alpha(x_{\beta_i})\right\}
        \\&
        =\bigcup_{i\leq m}\{n \in Z_\alpha:g_n(x_{\beta_i})\neq h_\alpha(x_{\beta_i}) \}\textrm{.}
 \end{split}
 \end{equation*} 
Thus, it is enough to show that for all $\beta<\omega_1$ we have $$\{n \in Z_\alpha: g_n(x_{\beta})\neq h_\alpha(x_{\beta})\}\in \I \textrm{.}$$ Fix $\beta<\omega_1$ and observe that $(h_{f_\alpha(n)}(x_{\beta}))_{n \in Z_\alpha} \overset{\I}{\longrightarrow}h_\alpha(x_{\beta})$. Indeed, if $\beta<\alpha$, then it follows from (P6) and if $\beta \geq \alpha$, then this is (P3). Therefore, since $U=\{h\in X: h(x_\beta)=h_\alpha(x_{\beta})\}$ is an open neighborhood of $h_\alpha$, we get
\begin{equation*}
    \begin{split}
        \{n\in Z_\alpha:g_n\notin U\}
        &
         =\{n \in Z_\alpha: g_n(x_{\beta})\neq h_\alpha(x_{\beta})\}
        \\&
        =\{n \in Z_\alpha: h_{f_\alpha(n)}(x_{\beta})\neq h_\alpha(x_{\beta})\}\in \I\textrm{.}
 \end{split}
 \end{equation*} 

\medskip

(iii): We show that $X$ is Hausdorff. Fix $h_\alpha,h_\beta\in X$ such that $h_\alpha\neq h_\beta$. We will consider 3 cases.

\smallskip

\textit{Case 1:} $\alpha,\beta \geq \omega$. Since $h_\alpha\neq h_\beta$ and $h_\alpha\restriction 2^\omega\setminus \tilde{X}=h_\beta\restriction 2^\omega\setminus \tilde{X}=0$ (by (P7)), there is $\gamma<\omega_1$ such that $h_\alpha(x_\gamma)\neq h_\beta(x_\gamma)$. Then $V_\alpha=\{h \in X : h_\alpha(x_\gamma)=h(x_\gamma)\}$ and $V_\beta=\{h \in X : h_\beta(x_\gamma)=h(x_\gamma)\}$ are neighborhoods of $h_\alpha$ and $h_\beta$, respectively, such that $V_\alpha\cap V_\beta=\emptyset$. 

\smallskip

\textit{Case 2:} $\alpha,\beta<\omega$. Then $h_\alpha(x_\alpha)=1$ and $h_\beta(x_\alpha)=0$, so $V_\alpha=\{h \in X : h_\alpha(x_\alpha)=h(x_\alpha)\}$ and $V_\beta=\{h \in X : h_\beta(x_\alpha)=h(x_\alpha)\}$ are the required disjoint neighborhoods of $h_\alpha$ and $h_\beta$, respectively.

\smallskip

\textit{Case 3:} $\beta<\omega\leq\alpha$. In this case $h_\beta(x_\beta)=1$ and $h_\alpha(x_\beta)=0$ (by (P7)). Similarly to the previous cases, $V_\alpha=\{h \in X : h_\alpha(x_\beta)=h(x_\beta)\}$ and $V_\beta=\{h \in X : h_\beta(x_\beta)=h(x_\beta)\}$ are the neighborhoods we need.

\medskip

(iv): Before showing that $X$ is not compact, we need to construct (by induction) an injective map $g:\omega_1\to\omega_1\setminus\omega$ such that $E_{g(\beta)}\supseteq^\star E_{g(\alpha)}$ whenever $\beta<\alpha<\omega_1$. Start with $g(0)=\omega$. If $n\in\omega$ and $g(n)$ is already defined, find $\omega\leq g(n+1)<\omega_1$ such that $E_{g(n)}=A_{g(n+1)}$ (this is possible as $E_{g(n)}\in\J^+$ by (P1)) and note that $E_{g(n+1)}\subseteq A_{g(n+1)}=E_{g(n)}$ (by (P1)) and $g(n+1)\neq g(i)$ for $i\leq n$ (as $g(n+1)=g(i)$ would imply $A_{g(i)}=A_{g(n+1)}=E_{g(n)}\subseteq^\star E_{g(i)}$, which contradicts $A_{g(i)}\setminus E_{g(i)}\in\J^+$). If $\omega\leq\alpha<\omega_1$ and all $g(\beta)$ for $\beta<\alpha$ are constructed, enumerate $\alpha=\{\beta_i: i\in\omega\}$ and define $B_n=\bigcap_{i\leq n}E_{g(\beta_i)}$ for all $n\in\omega$. Then $(B_n)_{n\in\omega}$ is a $\subseteq$-decreasing sequence of elements of $\J^+$ (as given $n\in\omega$ we have $E_{g(\max\{\beta_i:\ i\leq n\})}\subseteq^\star E_{g(\beta_i)}$ for all $i\leq n$, which gives us $\J^+\ni E_{g(\max\{\beta_i:\ i\leq n\})}\subseteq^\star B_n$ by (P1)). Using the fact that $\J$ is $P^+$, find $g(\alpha)\in[\omega,\omega_1)$ such that $A_{g(\alpha)}\setminus B_n$ is finite for all $n\in\omega$. Then for each $i\in\omega$ we have:
\[
E_{g(\alpha)}\setminus E_{g(\beta_i)}\subseteq A_{g(\alpha)}\setminus B_i\in\Fin,
\]
(by (P1)) and $g(\alpha)\neq g(\beta_i)$ for all $i\in\omega$ (as $g(\alpha)=g(\beta_i)$ would imply $A_{g(\beta_i)}\setminus E_{g(\beta_i)}=A_{g(\alpha)}\setminus E_{g(\beta_i)}\subseteq A_{g(\alpha)}\setminus B_i\in\Fin$ contradicting (P1)). Thus, $g(\alpha)$ is as needed.

We are ready to define an open cover of $X$ without a finite subcover. Let $U_\alpha=\{h\in X: h(x_{g(\alpha)})=0\}$ for all $\alpha<\omega_1$ and $V_n=\{h\in X:\ h(x_n)=1\}$ for all $n\in\omega$. Then $X=\bigcup_{\alpha<\omega_1} U_\alpha\cup\bigcup_{n\in\omega}V_n$. Indeed, it is clear that $h_n\in V_n$ for all $n\in\omega$ and if $h=h_\beta$ for some $\omega\leq\beta<\omega_1$ then there is $\alpha<\omega_1$ with $g(\alpha)>\beta$ (by injectivity of $g$), so $h\in U_\alpha$ by (P7). 

Fix now $\alpha_0<\ldots<\alpha_k<\omega_1$ and $n_0<\ldots<n_m<\omega$. We want to find some $h\in X\setminus(\bigcup_{i\leq k}U_{\alpha_i}\cup\bigcup_{j\leq m}V_{n_j})$. Since $\J^+\ni E_{g(\alpha_k)}\subseteq^\star\bigcap_{i\leq k}E_{g(\alpha_i)}$, we have $\bigcap_{i\leq k}E_{g(\alpha_i)}\in\J^+$, so there is some $p\in \bigcap_{i\leq k}E_{g(\alpha_i)}\setminus\{n_j:\ j\leq m\}$. Then $h_p\in X$, $h_p\notin\bigcup_{j\leq m}V_{n_j}$ and $h_p\notin U_{\alpha_i}$ for all $i\leq k$ (as $h_p(x_{g(\alpha_i)})=1$ by (P2) and $p\in E_{g(\alpha_i)}$). 

\medskip

Now we return to the construction of $x_\alpha\in 2^\omega$, $E_\alpha\subseteq\omega$, $Z_\alpha$ and $h_\alpha\in 2^{2^\omega}$ for $\omega\leq\alpha<\omega_1$.

Suppose that $\omega\leq\alpha<\omega_1$ and $x_\beta,E_\beta,Z_\beta$ as well as $h_\beta$ are defined for all $\omega\leq\beta < \alpha$ in such a way that (P1)-(P7) hold for them. We will divide the construction into two parts. At first we will construct $Z_\alpha$ and $h_\alpha$. Next, we will move on to $x_\alpha$ and $E_\alpha$.

\medskip

\textbf{First part of the construction}

\medskip

Before constructing $Z_\alpha$, we need to define auxiliary sets $\widehat{Z}_\alpha\in\I^+$. We will do this separately for $\alpha \in \nabla$ and for $\alpha \notin \nabla$.

\smallskip

\textit{Case 1:} If $\alpha \notin \nabla$, then put $\widehat{Z}_\alpha=f_\alpha^{-1}[[\omega,\omega_1)] \in \I^+$. Note that actually in this case we have  $\widehat{Z}_\alpha=f_\alpha^{-1}[[\omega,\alpha)]$ as $f_\alpha[\omega]\subseteq\alpha$.

\smallskip

\textit{Case 2:} If $\alpha \in \nabla$, then let $f \in \omega^\omega$ be defined by $$f(n)=\begin{cases}
    f_\alpha(n),  & \text{if } f_\alpha(n)<\omega, \\ 0, & \text{if } f_\alpha(n) \geq\omega \textrm{,}
\end{cases}$$
for all $n \in \omega$. From the fact that $\J \nleq_\BW\I$ it follows that there exists $\widetilde{Z}_\alpha\in \I^+$ such that for any given $B \subseteq f[\widetilde{Z}_\alpha]$ with $B \in \J^+$ there is $C$ for which $C\cap B\in\J^+$ and $\widetilde{Z}_\alpha \cap f^{-1}[C]\in\I$. We know that $f^{-1}[\{0\}]=f_\alpha^{-1}[\{0\}] \cup f_\alpha^{-1}[[\omega,\omega_1)]\in \I$. Therefore $\widehat{Z}_\alpha=\widetilde{Z}_\alpha\setminus f_\alpha^{-1}[[\omega,\omega_1)] \in \I^+$. Moreover, it holds that $f_\alpha[\widehat{Z}_\alpha]\subseteq\omega$ and for any given $B \subseteq f_\alpha[\widehat{Z}_\alpha]$ with $B \in \J^+$ there is $C$ for which $C\cap B\in\J^+$ and $\widehat{Z}_\alpha \cap f_\alpha^{-1}[C]\in\I$ (as $f\upharpoonright\widehat{Z}_\alpha=f_\alpha\upharpoonright\widehat{Z}_\alpha$).

\medskip

Now, we will define $Z_\alpha$. Fix any enumeration $\alpha=\{\beta_i:i \in \omega\}$. Let $B_\emptyset=\widehat{Z}_\alpha$ and for each $(s_0,\ldots,s_k)=s\in 2^{<\omega}$ define: 
$$B_s=\left\{n \in \widehat{Z}_\alpha: \forall_{i\leq k}\ h_{f_\alpha(n)}(x_{\beta_i})=s_i\right\} \textrm{.}$$ 
Observe that $\{B_s:s \in 2^{<\omega}\}$ is a family of subsets of $\widehat{Z}_\alpha$ such that:
\begin{enumerate}
\item $B_{\emptyset}=\widehat{Z}_\alpha$;
\item $B_s=B_{s^{\frown} (0)} \cup B_{s^{\frown} (1)}$ for all $s\in 2^{<\omega}$;
\item $B_{s^{\frown}  (0)} \cap B_{s^{\frown} (1)}=\emptyset$ for all $s\in 2^{<\omega}$.
\end{enumerate}
It follows from Theorem \ref{thm: hbw charakteryzacja} that there are $y \in 2^\omega$ and $Z_\alpha \in \I^+$ such that $Z_\alpha \setminus B_{y \upharpoonright n} \in \I$ for all $n\in\omega$. Without loss of generality we can assume that $Z_\alpha \subseteq \widehat{Z}_\alpha$. Hence, (P5) holds for such $Z_\alpha$. Now, let $h_\alpha(x_{\beta_i})=y_i$ for all $i \in \omega$ and $h_\alpha(x)=0$ for any $x\in 2^\omega\setminus\{x_\beta:\beta<\alpha\}$. It remains to show that (P6) and (P7) hold for such $Z_\alpha$ and $h_\alpha$.

We start with (P6). Fix $\beta<\alpha$. There is $i \in \omega$ such that $\beta=\beta_i$. Then 
\begin{equation*}
    \begin{split}
        \{ n \in Z_\alpha:h_{f_\alpha(n)}(x_{\beta_i})\neq h_\alpha(x_{\beta_i}) \}
        &
        =\{ n \in Z_\alpha: h_{f_\alpha(n)}(x_{\beta_i})\neq y_i\} 
        \\&
        \subseteq \{n \in Z_\alpha: \exists_{j \leq i}\ h_{f_\alpha(n)}(x_{\beta_j})\neq y_j   \}
        \\&
        =Z_\alpha \setminus B_{y\upharpoonright i+1} \in \I.
 \end{split}
 \end{equation*} 
Therefore (P6) holds for $Z_\alpha$ and $h_\alpha$.

\medskip

As $h_\alpha\upharpoonright2^\omega\setminus\{x_\beta:\beta<\alpha\}=0$, in order to get (P7) it remains to show that $h_\alpha\upharpoonright\{x_n:n\in \omega\}=0$. Fix $k \in \omega$. We need to consider two cases.

\smallskip

\textit{Case 1:} $\alpha \notin \nabla$. Then $f_\alpha[Z_\alpha]\subseteq[\omega,\omega_1)$ by (P5). Hence, for any $n \in Z_\alpha$ there is $\beta$ such that $f_\alpha(n)=\beta\in [\omega,\alpha)$ and therefore $h_{f_\alpha(n)}(x_k)=h_\beta(x_k)=0$ (by (P7) applied to $\beta$). Since by (P6) we know that $(h_{f_\alpha(n)}(x_k))_{n \in Z_\alpha} \overset{\I}{\longrightarrow} h_\alpha(x_k)$, we get $h_{\alpha}(x_k)=0$.

\smallskip

\textit{Case 2:} $\alpha \in \nabla$. Then $f_\alpha(n)<\omega$ for all $n \in Z_\alpha$ (by (P5)). Therefore, $h_{f_\alpha(n)}(x_k)=1$ only if $f_{\alpha}(n)=k$. Hence, $\{n \in Z_\alpha:h_{f_\alpha(n)}(x_k)=1\}\subseteq f_\alpha^{-1}[\{k\}]\in \I$. From $(P6)$ it follows that $h_\alpha(x_k)=0$.

\medskip

\textbf{Second part of the construction}

\medskip

Using the fact that $\J$ is nowhere maximal, find $E_\alpha^0\subseteq A_\alpha$ such that $E_\alpha^0\in \J^+$ and $A_\alpha\setminus E_\alpha^0\in \J^+$. Let $\nabla\cap(\alpha+1)=\{\delta_i:i \in \omega\}$. Before we will define $E_\alpha$, we need to construct (by induction) a sequence $(E_\alpha^i)_{i>0}$ such that the following conditions hold for all $i>0$:
\begin{itemize}
    \item[(a)] $E_\alpha^i \in \J^+$;
    \item[(b)] $E_\alpha^{i+1}\subseteq E_\alpha^i\subseteq E_\alpha^0 \subseteq A_\alpha$;
    \item[(c)] $Z_{\delta_i}\cap f^{-1}_{\delta_i}[E_\alpha^{i+1}]\in \I$.
\end{itemize}
Suppose that $i>0$ and $E_\alpha^i$ is already defined. We need to consider two cases.

\smallskip

\textit{Case 1:} If $E_\alpha^i\cap f_{\delta_i}[Z_{\delta_i}]\in\J$, then put $E_\alpha^{i+1}=E_\alpha^i \setminus f_{\delta_i}[Z_{\delta_i}]=E_\alpha^i \setminus(E_\alpha^i\cap f_{\delta_i}[Z_{\delta_i}])\in\J^+$ and observe that $Z_{\delta_i}\cap f^{-1}_{\delta_i}[E_\alpha^{i+1}]=\emptyset\in \I$.

\smallskip

\textit{Case 2:} If $E_\alpha^i\cap f_{\delta_i}[Z_{\delta_i}]\in\J^+$, then $\delta_i\in\nabla$ (otherwise $E_\alpha^i\cap f_{\delta_i}[Z_{\delta_i}]\subseteq\omega\cap [\omega,\delta_i)=\emptyset$ by (P5)). Applying condition (P5) to $\delta_i\in\nabla$ and the set $\J^+\ni B=E_\alpha^i\cap f_{\delta_i}[Z_{\delta_i}]\subseteq f_{\delta_i}[Z_{\delta_i}]$, we get $C$ such that $E_\alpha^{i+1}=C\cap B=C\cap (E_\alpha^i\cap f_{\delta_i}[Z_{\delta_i}])\in\J^+$ and $Z_{\delta_i}\cap f_{\delta_i}^{-1}[E_\alpha^{i+1}]\subseteq Z_{\delta_i}\cap f_{\delta_i}^{-1}[C]\in\I$.

\smallskip

This finishes the construction of the sequence $(E_\alpha^i)_{i\in\omega}$. Now, applying the fact that $\J$ is $P^+$ to this sequence, we obtain a set $\tilde{E}_\alpha\in\J^+$ such that $\tilde{E}_\alpha\setminus E_\alpha^i$ is finite for all $i\in\omega$. Without loss of generality we may assume that $\tilde{E}_\alpha\subseteq E_\alpha^0$. 

Observe that there are $2^\omega$ many $\J$-positive subsets of $\tilde{E}_\alpha$ (more precisely, using the fact that $\J$ is nowhere maximal, find $E\subseteq \tilde{E}_\alpha$ such that both $E$ and $\tilde{E}_\alpha\setminus E$ are $\J$-positive and observe that $E\cup E'\in(\J|\tilde{E}_\alpha)^+$ for every $E'\in\cP(\tilde{E}_\alpha\setminus E)$). Therefore, we can find some $E_\alpha\in (\J|\tilde{E}_\alpha)^+\setminus\{E_\beta:\beta<\alpha\}$. Then $E_\alpha\subseteq \tilde{E}_\alpha\subseteq E^0_\alpha\subseteq A_\alpha$, $E_\alpha\in\J^+$ and $A_\alpha\setminus E_\alpha\supseteq A_\alpha\setminus E^0_\alpha\in\J^+$, so condition (P1) is satisfied. Let $x_\alpha$ be as in (P2). It remains to check that (P3) and (P4) hold. 

\smallskip

At first we show (P3). Let $\omega \leq \beta \leq \alpha$. We know that $h_\beta(x_\alpha)=0 $ (by (P7) applied to $\beta$). We need to consider two cases.

\smallskip

\textit{Case 1:} $\beta \in \nabla$. Let $i\in\omega$ be such that $\beta=\delta_i$ and denote $F=\tilde{E}_\alpha\setminus E_\alpha^{i+1}$ (which is a finite set, by the definition of $\tilde{E}_\alpha$). Recall that $f_{\delta_i}[Z_{\delta_i}]\subseteq\omega$ (by (P5)). Thus, using (P2) and (c), we have:
\begin{equation*}
    \begin{split}
        \{n\in Z_{\delta_i}: h_{f_{\delta_i}(n)}(x_\alpha)=1\}
        &
        =\left\{ n \in Z_{\delta_i}: f_{\delta_i}(n)\in E_\alpha\right\} =Z_{\delta_i} \cap f^{-1}_{\delta_i}[E_\alpha]
        \\&
        \subseteq Z_{\delta_i} \cap f^{-1}_{\delta_i}[\tilde{E}_\alpha]\subseteq Z_{\delta_i} \cap \left(f^{-1}_{\delta_i}[E^{i+1}_\alpha]\cup f^{-1}_{\delta_i}[F]\right)
        \\&
        \subseteq \left(Z_{\delta_i}\cap f^{-1}_{\delta_i}[E^{i+1}_\alpha]\right)\cup f^{-1}_{\delta_i}[F]\in\I.
 \end{split}
 \end{equation*} 

\smallskip

\textit{Case 2:} $\beta \notin \nabla$. Then by (P5) we know that $f_\beta[Z_\beta]\subseteq[\omega,\beta)\subseteq[\omega,\alpha)$, so $h_{f_\beta(n)}(x_\alpha)=0$ for every $n\in Z_\beta$ (by (P7) applied to $f_\beta(n)$). It follows that $(h_{f_\beta(n)}(x_\alpha))_{n\in Z_\beta}$ is $\I$-convergent to $0$.

\smallskip

Now, it remains to show that (P4) holds as well. Fix $\beta<\alpha$. Since it is obvious that (P4) holds if $\beta<\omega$, it suffices to consider the case $\beta\geq\omega$. Then $E_\alpha\neq E_\beta$, therefore we can either find $i$ such that $i \in E_\alpha\setminus E_\beta$ or find $i \in E_\beta \setminus E_\alpha$. In the former case, $x_\alpha(i)=1\neq 0 =x_\beta(i)$. In the latter case, $x_\alpha(i)=0\neq 1=x_\beta(i)$.
\end{proof}

\begin{remark}
Since $\Fin^2$ is $\widehat{P}^+$ (by \cite[Proposition 2.6(2)]{zbMATH07566290}), Example \ref{ex:Fin^2} shows that in Theorem \ref{thm:space} we cannot replace $P^+$ with $\widehat{P}^+$.
\end{remark}

\begin{corollary}
\label{cor:spaces_for_Fsigma}
Assume CH. Let $\I$ and $\J$ be $\mathbf{F_\sigma}$ ideals. Then the following are equivalent:
\begin{itemize}
    \item[(a)] $\J\not\leq_{BW}\I$,
    \item[(b)] there is a $\BW(\I)$ space which is not $\BW(\J)$,
    \item[(c)] there is a Hausdorff non-compact $\BW(\I)$ space which is not $\BW(\J)$.
\end{itemize}
\end{corollary}

\begin{proof}
It follows from Theorem \ref{thm:space}, since every $\mathbf{F_\sigma}$ ideal is nowhere maximal, $P^+$ and has the $\hBW$ property (cf. Section \ref{sec:preliminaries}).
\end{proof}

\begin{corollary}\label{cor: spaces}
Assume CH. Then there exists:
\begin{itemize}
 \item[(a)] a Hausdorff sequentially compact (but not compact) space without the $\BW(\mathcal{W})$ property,
 \item[(b)] a Hausdorff sequentially compact (but not compact) space without the $\BW(\I_\frac{1}{n})$ property,
 \item[(c)] a Hausdorff $\BW(\I_\frac{1}{n})$ space which is not compact and not sequentially compact,
 \item[(d)] a Hausdorff non-compact $\BW(\mathcal{W})$ space without the $\BW(\I_\frac{1}{n})$ property,
 \item[(e)] a Hausdorff non-compact $\BW(\I_\frac{1}{n})$ space without the $\BW(\mathcal{W})$ property.
\end{itemize}    
\end{corollary}

\begin{proof}
The first two items follow from Corollary \ref{cor:spaces_for_Fsigma}, since Proposition \ref{prop:Fin_vs_tall}(a) implies that $\mathcal{W}\not\leq_\BW\Fin$ and $\I_{1/n}\not\leq_\BW\Fin$ (as $\mathcal{W}$ and $\I_{1/n}$ are tall). The last three items follow from Corollary \ref{cor:spaces_for_Fsigma} together with Propositions \ref{prop:idealy_nie_powyzej_Fin}(a) and \ref{prop:WvsSummable}. 
\end{proof}

\begin{remark}
Filip\'{o}w and Tryba in \cite[Question 2]{FT} asked whether there is a sequentially compact space that is not $\BW(\mathcal{W})$. Corollary \ref{cor: spaces}(a) gives a positive answer under CH.
\end{remark}

\section{Concluding remarks}
\label{sec:last}

In this section we discuss possible directions for further research.

By \cite[Section 2.7]{alcantara-phd-thesis}, there exists an ideal $\conv$ such that the following holds for every ideal $\I$:
$$[0,1]\in\FinBW(\I)\ \iff\ \conv\not\leq_K\I.$$
It would be interesting to find a critical (in the sense of $\leq_{BW}$) ideal for the property $[0,1]\in\BW(\I)$, that is, find some ideal $\J_{[0,1]}$ such that for every ideal $\I$ the following equivalence holds:
$$[0,1]\in\BW(\I)\ \iff\ \J_{[0,1]}\not\leq_{BW}\I.$$
We do not know whether such ideal exists. Perhaps the above equivalence requires some additional assumptions.

As we have seen, the space constructed in Theorem \ref{thm:space} is non-compact. Sometimes we can apply earlier results to find a compact Hausdorff $\BW(\I)$ space, which is not $\BW(\J)$: this is the case if it is known that $[0,1]\in\BW(\I)$ and $[0,1]\notin\BW(\J)$ (see \cite{MR2320288} for such examples). However, in general, the questions about existence of compact Hausdorff $\BW(\I)$ spaces, which are not $\BW(\J)$, seem to be more demanding.

Corollary \ref{cor: spaces} completely characterizes (under CH) the existence of $\BW(\I)$ spaces, which are not $\BW(\J)$, in the realm of $\mathbf{F_\sigma}$ ideals. For several other pairs of ideals, we are still able to answer questions about existence of such spaces:
\begin{itemize}
    \item By Corollary \ref{cor:idealy_powyzej_Fin}(c), every $\BW(\I_d)$ space is sequentially compact.
    \item By \cite[Example 3]{MR1181163}, there is a sequentially compact space (namely: $[0,1]$), which is not $\BW(\I_d)$.
    \item By \cite[Example 3.1]{MR2320288}, there is a sequentially compact space (namely: $[0,1]$), which is not $\BW(\nwd)$.
\end{itemize}
However, it is not hard to find pairs of well-known ideals, for which the question about existence of $\BW(\I)$ spaces, which are not $\BW(\J)$, remains open. For instance, we do not know the following:
\begin{itemize}
    \item Is there a $\BW(\nwd)$ space without the $\BW(\I_d)$ property?
    \item Is there a $\BW(\I_d)$ space without the $\BW(\nwd)$ property?
    \item Is there a $\BW(\nwd)$ space which is not sequentially compact?
\end{itemize}
Note that $\nwd\not\leq_{BW}\I_d$, $\I_d\not\leq_{BW}\nwd$ (both by Proposition \ref{prop:basic-properties}(d)) and $\Fin\not\leq_{BW}\nwd$ (by Proposition \ref{prop:idealy_nie_powyzej_Fin}(d)). 

Finally, the big open problem concerns the necessity of using CH in Theorem \ref{thm:space}: are there ideals $\I$ and $\J$, for which the existence of a $\BW(\I)$ space, which is not $\BW(\J)$, is independent of ZFC?

\bibliographystyle{amsplain}
\bibliography{references}

\end{document}